\documentclass[12pt,a4wide,reqno]{amsart}
\usepackage{amsfonts}
\title{
A Groshev Theorem for small linear forms  
}
\newtheorem{lem}{Lemma}
\newtheorem{cor}{Corollary}
\begin{document}

\newcommand{\KG}{Khintchine-Groshev} 
\newcommand{\K}{Khintchine} 
\newcommand{\W}{\widehat{W}}
\newcommand{\Q}{{\mathbb{Q}}}
\newcommand{\q}{\mathbf{q}}
\newcommand{\x}{\mathbf{x}}
\newcommand{\ip}{\cdot}

\newcommand{\hq}{\widehat{\mathbf{q}}}
\newcommand{\hx}{\widehat{\mathbf{x}}} 
\newcommand{\Vp}{V(\psi)}
\newcommand{\tp}{\widetilde{\psi}}
\newcommand{\R}{{\mathbb{R}}}
\newcommand{\T}{{\mathbb{T}}}
\newcommand{\N}{{\mathbb{N}}}
\newcommand{\Z}{{\mathbb{Z}}}

\address{Department of Mathematics\\University of York\\Heslington\\York
YO10 5DD}
\thanks{I thank my supervisor Maurice Dodson for suggesting this problem
and for his help and encouragement, and Sanju Velani and Simon
Kristensen for their helpful comments.  I am also grateful to the EPSRC
for financial support.}
\author{R.~S.~Kemble}

\maketitle

\section{Introduction}
In this paper the absolute value or distance from the origin analogue
of the classical \KG{} theorem~\cite{Sprindzuk} is established for a
single linear form with a `slowly decreasing' error function.  To
explain this in more detail, some notation is introduced.  Throughout
this paper, $m,n$ are positive integers i.e., $m,n\in \N$,
$\x=(x_1,\dots,x_n)$ will denote a point or vector in $\R^n$,
$\q=(q_1,\dots,q_n)$ will denote a non-zero vector in $\Z^n$ and
$$|\x|:=\max\{\vert x_1\vert, \dots, \vert x_n\vert \}=\Vert \x
\Vert_\infty
$$
will denote the height of the vector $\x$.  Let $\psi:\mathbb{N}\to
(0,\infty)$ be a (non-zero) function which converges to 0 at $\infty$.
The notion of a \emph{slowly decreasing function} $\psi$ is defined
in~\cite{Patterson} as a function for which given $c\in (0,1)$, there
exists a $K=K(c)>1$ such that $ \psi( ck) \le K \psi(k)$.  Of course
since $\psi$ is decreasing, $\psi(k)\le \psi(ck)$.  For any set $X$,
$\vert X \vert_n$ will denote the $n$-dimesnional Lesbegue measure of
$X$ (the suffix $n$ will usually be omitted; there should be no
confusion with the height of a vector).

\newtheorem{thm}{Theorem}

For convenience, the \KG{} theorem will be stated for a measurable set
$U\subseteq [0,1]^n$, where $n\in\N$.  The distance of a real number
$x$ from the integers $\Z$ is denoted by $\Vert x \Vert = \min
\{x-[x], \vert x-[x+1]\vert\}$, where $[x]$ is the integer part of
$x$. Let $W(U;\psi) = W([0,1]^{m\times n}, U;\psi)$ be the set of $\x\in
U$ such that the system of inequalities in $m$ linear forms in $n$
variables given by
  \begin{align*}
   \| q_{1}x_{11} + \dots + q_{n} x_{1n}\|& < \psi(|\q|)\\
 & \vdots \\
   \| q_{1}x_{m1} + \dots + q_{n} x_{mn}\|& < \psi(|\q|)
 \end{align*}
has infinitely many solutions in $\q\in\Z^n$. The \KG{} theorem states
that the
Lesbegue measure $\vert W(\psi) \vert$ depends on the convergence of a
sum involving $\psi(k), k \in \N$, as follows.
\begin{thm}
\label{thm:KG}
  The set $W(U;\psi), m,n \in \N$ has measure 
  \begin{equation*}
\vert W(U;\psi) \vert = \begin{cases}
    0 & \mbox{if } \displaystyle\sum_{k=1}^\infty k^{n-1} \psi^m(k) \mbox{ converges}, \\
    \vert U \vert & \mbox{if } \displaystyle\sum_{k=1}^\infty k^{n-1} \psi^m(k)
   \text{ diverges and $\psi(k)$ is decreasing} \\
 &\text{                 for $n=1$ or $2$}.
\end{cases}
  \end{equation*}
\end{thm}

We note that if we consider $C\psi$ for some positive constant $C$, the
conclusion does not change, since the constant does not affect the
convergence of the sum $\sum_{k=1}^\infty k^{n-1} C^m\psi^m(k)$.

Let the set of points $\x\in\R^n$ for which the linear form
\begin{equation*}
\q\ip\x   = q_{1}x_1 + \dots + q_{n} x_n  
\end{equation*}
 satisfies
\begin{equation}
\label{eq:1}
\vert   \q\ip\x\vert=\vert q_1x_1+\dots+q_nx_n\vert <\psi(\vert \q\vert)
\end{equation}
for infinitely many $\q\in\Z^n$ be denoted by $V(\R^n,\psi)$.  The set
$V(\R^n,\psi)$ is the absolute value analogue of the \KG{} theorem for
a single linear form. It will be shown that the Lebesgue measure of
$V(\R^n,\psi)$
depends on the convergence or divergence of the sum
\begin{equation}
 \label{eq:critsum}
  \sum_{k=1}^\infty k^{n-2}\psi(k).
\end{equation}
For convenience, the subset
\begin{align*}
V([0,1]^n;\psi) &= V(\R^n;\psi)\cap [0,1]^n \\
&=\{ \x \in [0,1]^n \colon \vert
\mathbf{q} \ip 
\mathbf{x} \vert < \psi (\vert \q \vert )\ \mbox{for infinitely many} \ \q \}
\end{align*}
is considered and will be called simply $V(\psi)$.  The
analogue of the \KG{} theorem for $V(\psi)$ is now
stated.  Since $V(\psi)=\{0\}$ when $n=1$, from now on take $n\ge 2$.

\begin{thm}
For $n\ge 2$,
  $$
  \vert V(\psi) \vert = \begin{cases}
    0 & \mbox{if } \displaystyle\sum_{k=1}^\infty k^{n-2} \psi(k) \mbox{ converges}, \\
    1 & \mbox{if } \displaystyle\sum_{k=1}^\infty k^{n-2} \psi(k)
    \mbox{ diverges and $\psi(k)$ is slowly decreasing}.
\end{cases}
$$
\end{thm}

\section{Proof of theorem}
\label{sec:pot}

The set
\begin{equation}
  \label{eq:Bd}
B_\delta (\q) = \{ \x \in [0,1]^n : \vert\q
\ip \x\vert < \delta\}  
\end{equation}
is a neighbourhood of the resonant set
\begin{equation}
  \label{eq:reset}
  R(\q)= \{ \x \in [0,1]^n : \q \ip \x =0\}
\end{equation}
and is a `thickened' hyperplane or parallelepiped of volume
\begin{equation}
\label{eq:ball}
\vert B_\delta(\q)\vert \leq 2\delta / \vert \q \vert.  
\end{equation}
The error arises when two or more coordinates of $\q$ are equal
to the height of $\q$.  This error, when it
occurs, is negative and therefore has no consequence in any questions
of convergence.  It is readily verified the set $V(\psi)$ can be
expressed in the following `limsup' form:
\begin{equation}
  \label{eq:limsup}
V(\psi) 
= \bigcap_{N=1}^{\infty} \bigcup_{k=N}^{\infty}
\bigcup_{\vert \q\vert =k} B_{\psi(\vert \q \vert)}(\q) 
\subseteq \bigcup_{q=N}^{\infty} B_{\psi(\vert \q \vert)}(\q)  .
\end{equation}
The proof falls into two parts, depending on whether the sum
\eqref{eq:critsum} converges or diverges.

\subsection{Convergence case }

Suppose the sum \eqref{eq:critsum} converges.  
Replacing $\delta$ with $ \psi (\vert \q \vert)$ and summing over all
non zero $\q$, we have by~\eqref{eq:ball},
\begin{align*}
\sum_{\q} \vert B_{\psi(\vert \q \vert)}(\q) \vert &\leq \sum_{\q} \frac{2 \psi (\vert \q \vert)}{\vert \q \vert} =
2\sum_{k=1}^{\infty} \sum_{\vert \q \vert = k} \frac {\psi(\vert \q
\vert)}{\vert \q \vert} \\
&= 2 \sum_{k=1}^{\infty} \frac{\psi (k)}{k} \sum_{\vert \q \vert =k} 1
\\
&= 2^{n+1}n \sum_{k=1}^{\infty} \psi (k) k^{n-2}+O(\sum_{k=1}^{\infty} \psi (k) k^{n-3}).
\end{align*}

A Borel-Cantelli Lemma argument is now applied to
\eqref{eq:limsup} to complete
the proof of the convergence case.  
By \eqref{eq:limsup}, for each $N=1,2,\dots,$
\begin{align*}
\vert V(\psi) \vert &\leq \sum_{k=N}^{\infty} \sum_{\vert \q \vert =k}
\vert B_{\psi(k)}(\q) \vert \ll \sum_{k=N}^{\infty} \sum_{\vert \q \vert
=k} \frac{\psi(\vert \q \vert)}{\vert \q \vert} \\
& \ll \sum_{k=N}^{\infty} \psi(k)k^{n-2}
\end{align*}
by \eqref{eq:ball}. 
But the sum~\eqref{eq:critsum} converges, whence the tail tends to 0
as $N\to \infty$, and so $\vert V(\psi) \vert = 0$.

It is readily verified that the argument used here extends to systems
of $m$ linear forms so that if the sum $\sum_{k=1}^\infty
k^{n-2}\psi^m(k)$ converges, the set of points in $\R^{m\times n}$ for
which the inequalities  

\begin{align*}
   | q_{1}x_{11} + \dots + q_n x_{1n}|& < \psi(|\q|)\\
 & \vdots \\
   | q_{1}x_{m1} + \dots + q_{n} x_{mn}|& < \psi(|\q|)
\end{align*}
have infinitely many solutions in $\q\in\Z^n$ is of measure zero.

\subsection{Divergence case }
Now assume that the sum~\eqref{eq:critsum} diverges. Since $\psi$ is slowly decreasing, there
exists a constant $C>1$ such that for all $k\in \N$,
\begin{equation}
  \label{eq:sldecr}
 \psi(k/n) \leq C\psi(k).   
\end{equation}

Define the set $F_j=\{\x \in [0,1]^n : x_j =1\}$, which
will be referred to as the $j$th face of the hypercube $[0,1]^n$.  For
points in the set $V(\psi) \cap F_n$, the inequality~\eqref{eq:1}
reduces to
\begin{equation}
\label{eq:face}
\vert q_1 x_1 + q_2x_2 + \dots + q_{n-1} x_{n-1} + q_n \vert 
< \psi ( \vert \q \vert )  .
\end{equation}

It is shown that $\vert V(\psi)\cap F_n \vert = 1$ by proving
$$
W(\R^{1\times (n-1)},[0,1]^{n-1};\psi/C) \times \{1\} \subseteq V(\psi) \cap F_n.
$$  
Then the sum
\begin{equation*}
  \sum_{k=1}^\infty K\psi(k)k^{n-2}=\infty
\end{equation*}
for any constant $K>0$.
By the
$(n-1)$-dimensional Groshev theorem, the $(n-1)$ dimensional Lesbegue
measure is given by
$$
\vert W([0,1]^{1\times (n-1)}, [0,1]^{n-1};\psi/C) \vert =1.
$$

First we consider the case $j=n$.  To discuss the validity of \eqref{eq:face}, let $\hq = (q_1,
\dots, q_{n-1})$ and define $\hx$ similarly. Note that $|\q|\ge
|\hq|$.  
Now let $\hx \in W([0,1]^{1\times (n-1)},[0,1]^{n-1};\psi/C).$  Then by definition
there exist infinitely many $\hq = (q_1, \dots, q_{n-1}) \in \Z^{n-1}$ such that
\begin{equation*}
\left\| \hq \ip \hx \right\| < \frac1{C}\psi ( \vert \hq \vert ), 
\end{equation*}
that is, such that
\begin{equation*} 
\vert q_1x_1 + \dots +q_{n-1}x_{n-1} + q_n \vert 
< \frac1{C}\psi ( \vert \hq \vert )
\end{equation*}
for infinitely many $\hq \in \Z^{n-1}, q_n\in \Z$.  Now suppose $\vert \hq \vert = \vert \q \vert$ for infinitely many
$\q$.  Then since $C>1$, it is clear that $(\hx,1) \in V(\psi) \cap F_n.$
Otherwise, suppose $\vert \hq \vert = \vert \q \vert$ for only
finitely many $\q.$

Then $\vert \q \vert > \vert \hq \vert$ for all $\q$ with $\vert \q
\vert$ sufficiently large. It follows that $|\q| =|q_n|$ for all $\q$
with $\vert \q \vert$ sufficiently large. Moreover there exists a $q_n
\in \Z$ such that,
\begin{equation*}
  \vert q_1x_1 + \dots +q_{n-1}x_{n-1} + q_n \vert < 1,
\end{equation*}
i.e.,
\begin{equation*}
\vert \q\ip\x \vert < 1,
\end{equation*}
where $\x=(x_1,\dots,x_{n-1}, x_n)$.  
Thus there exists $j, 1 \leq j \leq n-1$ such that
\begin{equation}
  \label{eq:ineq}
\vert \hq \vert \geq \vert q_j \vert > \frac{\vert \q \vert}{n}
= \frac{\vert q_n \vert}{n}.  
\end{equation}
 For if not, then for each $j=1,\dots,n-1$,  $|q_j|\le |\q|/n=|q_n|/n$ and so  
\begin{equation*}
\vert \hq \ip \hx \vert=\vert q_1x_1+\dots+q_{n-1}x_{n-1} \vert\le
\frac{(n-1)|q_n| }{n}, 
\end{equation*}
whence
\begin{equation*}
 \vert \hq \ip \hx+q_n \vert \ge  \vert q_n \vert - \vert \hq \ip \hx
 \vert \ge  \vert q_n \vert -\left(1-\frac1{n}\right)|q_n|=\frac{|q_n|}{n}>1.
\end{equation*}
Therefore
$$
\vert \q\ip\x \vert = 
\vert \hq \ip \hx + q_n \vert< \frac1{C}\psi ( \vert \hq \vert
) \leq \frac1{C} C\psi(\frac{\vert \q \vert}{n}) < \psi(\vert \q \vert)
$$
by~\eqref{eq:sldecr} for infinitely many $\q$, so $\x \in
V(\psi) \cap F_n,$ and 
\begin{equation*}
W([0,1]^{(n-1)},[0,1]^{n-1};\psi/C) \times \{1\}
\subseteq V(\psi) \cap F_n.  
\end{equation*}
The argument for $F_j$ is similar, $1\leq j < n.$
\begin{lem} \label{lem:one}
Suppose  $t \in [0,1]$ and $\x\in V(\psi)$.  Then $t\x \in V(\psi)$.
\end{lem}

\begin{proof}
Clearly $\mathbf{0} \in V(\psi)$, so the implication is true for $t=0$. Suppose $t>0$.
By definition of $\x=(x_1,\dots,x_n)\in[0,1]^n$, 
there exist infinitely many $\q$ such that
\begin{equation*}
\vert q_1 x_1 + \dots + q_n x_n \vert < \psi(\vert \q \vert). 
\end{equation*}
Since $\psi(\vert \q \vert) > 0$,
\begin{equation*}
\vert q_1 tx_1 + \dots + q_n tx_n \vert = t\vert q_1 x_1 + \dots + q_n x_n
\vert 
\leq t \psi(\vert \q \vert) < \psi (\vert \q \vert).
\end{equation*}
Hence $t\x \in V(\psi)$ for all $t \in [0,1].$  
\end{proof}
%\pagebreak
\begin{cor} 
\label{cor:one}
If $\x\in F_j \cap
V(\psi)$, then $t\x \in V(\psi)$.
\end{cor}

For each $j=1,\dots,n$ let $P_j$ be the `pyramid' with vertex at the
origin and base $F_j$, i.e.
\begin{equation*}
P_j = \{\x \in [0,1]^n : \vert \x \vert = x_j\} = \{t\x: t \in [0,1], \x
\in F_j \}. 
\end{equation*}
Note that $F_j \subset P_j.$  Also for each $U \subseteq [0,1]^{n-1}$, 
write 
\begin{align*}
\W(U;\psi) &=W([0,1]^{n-1},U;\psi)\\
& = \{\mathbf{y} \in U : \Vert \q\ip\mathbf{y}
\Vert < \psi(\vert \q \vert) \mbox{ for infinitely many } \q\}.
\end{align*}
Then by Theorem~\ref{thm:KG}, $\vert \W(U;\psi)\vert=\vert
U\vert$ since~\eqref{eq:critsum} diverges.

Now 
\begin{equation*}
  \vert U \vert \geq \vert F_j \cap ( V(\psi) \cap U ) \vert \geq 
\vert \W(U;\psi/C) \vert 
= \vert U \vert .
\end{equation*}

The characteristic function of the set S will be denoted $\chi_S.$
\begin{lem} 
\label{lem:two}
For each $t\in [0,1]$ and $\x \in [0,1]^n,$
 $$
\chi_{_{V(\psi)\cap P_j}} (t\x) \geq \chi_{_{V(\psi)\cap F_j}} (\x).
$$
\end{lem}
\begin{proof}
If $\chi_{_{V(\psi)\cap F_j}} (\x) = 1$, then $\chi_{_{V(\psi)\cap P_j}} (t\x)
= 1$ by Lemma 1.
If $\chi_{_{V(\psi)\cap F_j}} (\x) = 0$, then $\chi_{_{V(\psi)\cap P_j}} (t\x)
= 0 \mbox{~or~} 1$.
\end{proof}

\begin{lem} \label{lem:three}
For each $j=1,\dots,n,$
 $$ \vert V(\psi) \cap P_j \vert = 1/n. $$
\end{lem}
\begin{proof}
Take $j=n.$ Then
\begin{align*}
\vert V(\psi)\cap P_n \vert &=\int_{0}^{1}\int_{0}^{t} \dots
\int_{0}^{t} \chi_{_{V(\psi)\cap P_n}} (tx_1, \dots,tx_{n-1},t) dx_1 \dots
dx_{n-1} dt \\
&\geq \int_{0}^{1} \left[ \int_{0}^{t} \dots
\int_{0}^{t} \chi_{_{V(\psi)\cap F_n}} (\x) dx_1 \dots dx_{n-1} \right] dt 
\end{align*}

Take $U=[0,t]^{n-1}$. Then the inner multiple integral can be evaluated as
\begin{align*}
   \int_{0}^{t} \dots \int_{0}^{t} \chi_{_{V(\psi)\cap F_n}} (\x)
    dx_1 \dots dx_{n-1} & = \vert V(\psi) \cap F_n \cap
  [0,t]^{n-1} \vert \\ & =
  \vert \W([0,t]^{n-1},\psi) \vert \\
  &= \vert [0,t]^{n-1} \vert \\
  &= t^{n-1}.
\end {align*}
Therefore 
\begin{equation*}
\vert V(\psi) \cap P_n \vert \geq \int_{0}^1 t^{n-1} dt = 1/n.
\end{equation*}
The cases for $j \neq n$ are similar. 

It follows that
\begin{align*}
\vert V(\psi) \vert &= \bigcup_{j=1}^n \vert V(\psi) \cap P_j \vert 
\geq \sum_{j=1}^n \vert V(\psi) \cap F_j \vert =n.1/n = 1,
\end{align*}
whence $\vert V(\psi) \vert = 1$ as claimed.  
\end{proof} 
A direct proof would be more satisfactory but there is a similar
configuration of curved lines in the plane for which the analogous
measure result does not hold (see~\cite{Bugeaud}).

Unlike the case of convergence, the arguments in the divergence case
do not extend to the case of more than one linear form.
Now, in~\cite{Dickinson}, H. Dickinson considered the system of linear
forms in the case when $\psi(k) = k^{-\tau}$ and showed that the
Hausdorff dimension is given by
$$
 \dim_{\text H} V = \begin{cases}
      (m-1)n + \frac{m}{\tau +1}, & \tau > (m/n)-1 \\
      mn, & 0 < \tau \leq (m/n)-1.
\end{cases}
$$
The argument establishing the upper bound is essentially the same
as the argument in the case for convergence.  The argument for the
lower bound uses ubiquity and there is a close connection between this
and the divergence case of Khintchine's theorem. However, Dickinson's
approach does not appear to be adaptable in a straightforward manner
to more than one linear form in the divergence case.

Although less powerful, the methods used in this paper are concise and
direct. They can also be used to provide a lower bound for the number
of solutions from the formula for the classical case
(see~\cite{Schmidt} or~\cite[Chap.~1, \S7]{Sprindzuk}) applied to the
$(n-1)$-dimensional faces of the unit cube.  Then using the same
argument as the divergence case to overcome the difficulty that arises
when 
$|\q|=|q_n|$ discussed in \S2.2, it can be verified that the number of solutions
$N(\x,Q)$ satisfies
\begin{equation*}
N(\x,Q) \gg \sum_{k=1}^Qk^{n-2}\psi(k).  
\end{equation*}

\providecommand{\bysame}{\leavevmode\hbox to3em{\hrulefill}\thinspace}

\end{document}